\documentclass[12pt,reqno]{article}%
\usepackage[pdftex]{graphicx}
\usepackage{amsfonts}
\usepackage{amsmath}
\usepackage{amsmath, amscd, amsfonts, amsthm, tikz, times}

\newcounter{conjecture}
\newcounter{remark}
\newcounter{corollary}

\newenvironment{corollary}{\medskip{\bf Corollary \thecorollary.}
\addtocounter{corollary}{1}\em}{\rm}
\newtheorem{theorem}{Theorem}
\newtheorem{lemma}{Lemma}
\newtheorem{proposition}{Proposition}

\def \be{\begin{equation}}
\def \ee{\end{equation}}
\def \bt{\begin{theorem}}
\def \et{\end{theorem}}
\def \bc{\begin{corollary}}
\def \ec{\end{corollary}}
\def \bea{\begin{eqnarray}}
\def \eea{\end{eqnarray}}
\def \bas{\begin{eqnarray*}}
\def \eas{\end{eqnarray*}}
\def \bl{\begin{lemma}}
\def \el{\end{lemma}}

\def \vski{\vspace{12pt}}

\def \ZZ{{\mathbb Z}}

\def \({\left(}
\def \){\right)}

\def \nn{\nonumber}

\def \bc{\begin{center} }
\def \ec{\end{center} }
\def \bs{\begin{slide} }
\def \es{\end{slide} }

\def\square{{\vcenter{\vbox{\hrule height.3pt
         \hbox{\vrule width.3pt height5pt \kern5pt
            \vrule width.3pt}
         \hrule height.3pt}}}}

\newcounter{cccases}
\begin{document}

\title{Remarks on the recurrence and transience of non-backtracking random walks}

\author{Paul Jung$^{\,\rm 1}$, Greg Markowsky$^{\,\rm 2}$\\
{\small {\tt pauljung@kaist.ac.kr} ~~ {\tt greg.markowsky@monash.edu}}\\
{\footnotesize{$^{\rm 1}$Department of Mathematics, KAIST, Republic of Korea}}\\
{\footnotesize{$^{\rm 2}$Department of Mathematics,  Monash University, Australia}}}




\maketitle


\begin{abstract}
A short proof of the equivalence of the recurrence of non-backtracking random walk and that of simple random walk on regular infinite graphs is given. It is then shown how this proof can be extended in certain cases where the graph in question is not regular.
\end{abstract}



\vski
\vski

{\it Non-backtracking random walks} (abbr. NBRW) is a random process defined on the vertices of a graph $G$ by the transition probabilities
\begin{equation} \label{}
\begin{gathered}
P(X_1 = y|X_0=x) = \left \{ \begin{array}{ll}
\frac{1}{deg(x)} & \qquad  \mbox{if } y \sim x  \\
0 & \qquad \mbox{if } y \not \sim x\;,
\end{array} \right.\\
P(X_{n+1} = y|X_n=x_2,X_{n-1} = x_1) = \left \{ \begin{array}{ll}
\frac{1}{deg(x_2)-1} & \qquad  \mbox{if } y \sim x_2 \mbox{ and } y \neq x_1 \\
0 & \qquad \mbox{if } y \not \sim x_2 \mbox{ or } y=x_1\;.
\end{array} \right.
\end{gathered}
\end{equation}

\vski

Recently NBRWs  have received much attention in the scientific literature due to their connection to spectral methods for community detection, \cite{krzakala2013spectral, bordenave2018nonbacktracking, abbe2017community} and also because they mix faster than the corresponding simple random walks on various graphs $G$, \cite{alon2007non, lee2012beyond, kemp2}. They have also played a role in the spectral analysis of random matrices \cite{sodin2007random}.

\vski

Consideration of the NBRW on $\ZZ^d$ dates back to \cite[Section 5.3]{madras1996self} and the process was also analyzed in \cite{fitzner2013non} where the Green's function and functional central
limit theorem were studied.
In \cite{kemp}, the question of recurrence of non-backtracking random walk on the integer lattice was settled: the walk is transient on $\ZZ$ (trivially) and on $\ZZ^d$ with $d \geq 3$, and recurrent on $\ZZ^2$. Following this, in \cite[Prop 1.1]{hermon}, the following extension was proved.

\begin{theorem} \label{bigguy}
If $G$ is a regular infinite graph of degree $k \geq 3$, then non-backtracking random walk is recurrent on $G$ if, and only if, simple random walk is recurrent on $G$.
\end{theorem}

The result in \cite{kemp} was proved by a rather intricate counting argument, and in \cite{hermon} the universal cover of $G$ is used to prove Theorem \ref{bigguy} in amongst a larger work on reversibility. Each of these works contain a number of results of independent interest, however if we are only interested in Theorem \ref{bigguy} there is a very short and simple probabilistic coupling-type proof available, as we now present.

\vski

\begin{proof}
We begin with an algorithm which modifies deterministic sequences of the vertices of $G$. Let us suppose $(x_0,x_1,x_2, \ldots)$ is such a sequence. Our algorithm will move along the sequence and essentially remove the backtracking contained in the sequence. The algorithm is as follows:
\begin{enumerate}
	\item[1.] If we are at position 0, move to the right.
	\item[2.] If we are at position $n>0$, compare $x_{n-1}$ with $x_{n+1}$, and
	\item[2a.]  if $x_{n-1} \neq x_{n+1}$, then move to the right,
	\item[2b.] if $x_{n-1} = x_{n+1}$, then erase $x_n, x_{n+1}$ and shift the remaining part of the sequence two positions to the left to close the gap (so that portion now reads $(\ldots, x_{n-1}, x_{n+2}, \ldots)$). Next, move to the left.
	\item[3.] Repeat.
\end{enumerate}
The reason for moving left in step (2b) is to check whether the erasure has introduced a new backtracking. Note that, if our initial sequence is such that the algorithm eventually leaves any given finite subset of positions forever, then the output of this algorithm will be a sequence which contains no backtracking.

\vski

Let us apply this algorithm to simple random walk (abbr. SRW) on the vertices of $G$. If $(X_n)_n$ denotes SRW, then we may apply our algorithm to the random sequence $(X_0, X_1,X_2, \ldots)$. It is straightforward to verify that the regularity of $G$ implies that the output from the algorithm is NBRW, provided only that the algorithm eventually leaves any finite subset of positions forever. Examining the (now random) movements of the algorithm, notice that they have independent increments and that at any position $n>0$, there is probability $\frac{1}{k}$ of moving to the left and probability $\frac{k-1}{k}$ of moving to the right. These movements constitute a birth-death chain, which is well known to be transient in this case since $\frac{k-1}{k}>\frac{1}{2}$ (see for instance \cite{norris}).

\vski

It is now immediate that a graph $G$ transient for SRW is also transient for NBRW, since our algorithm can in no way turn finitely many visits to any vertex in SRW into infinitely many for NBRW. To see that a $G$ recurrent for SRW is also recurrent for NBRW is slightly more subtle. Fix a vertex $v$ and note that the event of a given visit to $v$ of SRW, say $X_n=v$, not being erased by our algorithm contains the event $E$ that our birth-death chain visits the `current position' of $X_n$ for the last time (the current position of $X_n$ takes into account any shifts caused by step (2b) of the algorithm up to time $n$, and thus the current position can be any $k\le n$, such that $k$ has the same parity as $n$). Note that $P(E)>0$. Consider now, successive visits to $v$ by $(X_n)_n$, and the associated events $\{E_{m_n}\}_n$ of the birth-death chain visiting the current positions for the last time. These events are independent, thus infinitely many visits to any vertex in $G$ for SRW must remain infinitely many in NBRW.
\end{proof}

{\bf Remarks} \begin{itemize} \label{}

\item It is satisfying to note the reason that the proof fails for $k=2$: the resulting birth-death chain produced by the algorithm is SRW on the integers, which is well known to be recurrent; therefore our algorithm fails and NBRW cannot be produced in this fashion.

\item The condition that $G$ is regular is necessary for the proof (although it can be weakened somewhat; see below). To see this, note that if $G$ is not regular then the following situation may arise. Let us suppose $v$ is a vertex of degree 3 adjacent to $x,y,$ and $z$. If NBRW reaches $v$ via $x$ then it should have equal probabilities of passing next to $y$ and $z$. However if the degree of $y$ is higher than that of $z$ then a passage of SRW from $v$ to $z$ is more likely to be erased by what follows than one from $v$ to $y$. When we apply our algorithm to SRW, then, our output is a random process without backtracking, but it is not equal in distribution to NBRW.
\end{itemize}

Incidentally, we may adjust the method of proof used in Theorem \ref{bigguy} in order to handle the following situation.

\begin{proposition} \label{bigguy2}
Suppose $G$ is an infinite graph such that every vertex has either degree $k_1$ or $k_2$, with $k_1 > k_2 \geq 2$. Suppose further that every vertex of degree $k_1$ is adjacent only to vertices of degree $k_2$, and every vertex of degree $k_2$ is adjacent only to vertices of degree $k_1$. Then non-backtracking random walk is recurrent on $G$ if, and only if, simple random walk is recurrent on $G$.
\end{proposition}

The only adjustment to our previous proof is to note that the birth-death chain produced by our algorithm has two different probabilities; for instance, if we start at a vertex of degree $k_1$ then the resulting birth-death chain moves to the right with probability $\frac{k_1-1}{k_1}$ at even positions in the sequence, and $\frac{k_2-1}{k_2}$ at odd positions. Nevertheless, this birth-death chain is still easily seen to be transient by the methods in \cite{norris}, and the rest of the proof persists unchanged.

\vski

As a bit of an aside, we may also give a partial answer to an intriguing question posed in \cite{hermon}. Question 1.11 in that work is as follows:

\vski

{\it Let $G$ be a connected graph of bounded degree such that the length of any path of vertices of degree 2 is bounded by a finite constant $L > 0$. Is
it the case that the SRW on G is transient iff the NBRW on G is transient?}

\vski

We will show that the answer is in the affirmative provided that the vertices of the graph have only one possible degree other than 2. In other words, we have the following proposition.

\begin{proposition} \label{ziyi}
Let $G$ be a connected graph where there is a constant $k>2$ such that every vertex has either degree 2 or $k$. Suppose further that the length of any path of vertices of degree 2 is bounded by a finite constant $L > 0$. Then non-backtracking random walk is recurrent on $G$ if, and only if, simple random walk is recurrent on $G$.
\end{proposition}

\begin{proof}
We appeal to the theory of electric networks and their connections to random walk (\cite{snell} is one classic reference on this topic). Let $V$ denote the set of all vertices of $G$ of degree $k$, and let $G'$ be a graph with vertices $V$ and with two vertices adjacent if they are adjacent in $G$ or if there is a path of vertices of degree 2 connecting them in $G$. Assign to each edge in $G'$ a resistance of one if the edge exists in $G$ or a resistance equal to the length of the corresponding path of vertices of degree 2 in $G$, otherwise. Suppose we start a SRW $(X_n)_n$ on $G$ at a point in $V$ (recurrence and transience are not affected by the initial point). Define a sequence of stopping times by $\tau_0 =0$ and $$\tau_n = \min\{j > \tau_{n-1}: X_j \in V\}\quad  \text{ for } n \geq 1.$$ It then may be checked that the process $(X_{\tau_0}, X_{\tau_1}, X_{\tau_2}, \ldots)$ on $V$ is equal in distribution to weighted random walk (abbr. WRW) on $G'$, as defined in \cite{snell}. We will prove the following set of equivalences.

\begin{equation} \label{} \nn
\begin{split}
SRW & \mbox{ on $G$ is recurrent} \stackrel{(1)}{\Longleftrightarrow} WRW \mbox{ on $G'$ is recurrent} \stackrel{(2)}{\Longleftrightarrow} SRW \mbox{ on $G'$ is recurrent} \\
&\stackrel{(3)}{\Longleftrightarrow} NBRW \mbox{ on $G'$ is recurrent} \stackrel{(4)}{\Longleftrightarrow} NBRW \mbox{ on $G$ is recurrent}
\end{split}
\end{equation}

\begin{itemize} \label{}

\item[(1)] follows since $X_{\tau_n}$ is recurrent precisely when $X_n$ is.

\item[(2)] is immediate from Theorem 2.4.3 of \cite{snell}, which states that recurrence is equivalent for WRW and SRW (which can be realized as having all resistances set to 1) provided that we have a finite upper and lower bound on the resistances, as we do in this case due to the existence of the constant $L$.

\item[(3)] is immediate from Theorem \ref{bigguy}.

\item[(4)] follows by noting that if we let $X_n$ be NBRW on $G$ and define the stopping times $\tau_n$ as above, then the process $X_{\tau_0}, X_{\tau_1}, X_{\tau_2}, \ldots$ on $V$ is equal in distribution to NBRW on $G'$, as the process $X_n$ may not reverse directions on the paths of vertices of degree 2.

\end{itemize}

These equivalences complete the proof of the proposition.
\end{proof}

This method of proof can be further extended to the following situation, with Proposition \ref{bigguy2} taking the place of Theorem \ref{bigguy} where required. Details are left to the reader.

\begin{proposition} \label{}
Let $G$ be a connected graph where there are constants $2<k_1<k_2$ such that every vertex has degree $2, k_1$, or $k_2$. Suppose further that the length of any path of vertices of degree 2 is bounded by a finite constant $L > 0$, and that the graph $G'$ formed as in the proof of Proposition \ref{ziyi} is of the form required in Proposition \ref{bigguy2}. Then non-backtracking random walk is recurrent on $G$ if, and only if, simple random walk is recurrent on $G$.
\end{proposition}

\section*{Acknowledgments}

P. Jung was supported in part by the (South Korean) National Research Foundation grant N01170220.

\bibliographystyle{alpha}
\bibliography{CAbib}

\end{document}